\documentclass[a4paper,11pt]{amsart}

\usepackage[T1]{fontenc}

\usepackage{setspace}
\usepackage{times}
\usepackage{mathrsfs}
\usepackage{amsfonts,amssymb,amsmath}
\usepackage[all]{xy}

\usepackage{hyperref}

\newtheorem{thm}{Theorem}[section]
\newtheorem{defn}[thm]{Definition}
\newtheorem{lemma}[thm]{Lemma}
\newtheorem{prop}[thm]{Proposition}

\theoremstyle{remark}
\newtheorem{rmk}[thm]{Remark}

\newcommand{\bC}{\mathbb{C}}
\newcommand{\bL}{\mathbb{L}}

\newcommand{\bR}{\mathbb{R}}
\newcommand{\bV}{\mathbb{V}}
\newcommand{\bZ}{\mathbb{Z}}

\newcommand{\cC}{\mathcal{C}}
\newcommand{\cE}{\mathcal{E}}
\newcommand{\cF}{\mathcal{F}}
\newcommand{\cG}{\mathcal{G}}

\newcommand{\cO}{\mathcal{O}}
\newcommand{\cN}{\mathcal{N}}

\newcommand{\red}{\mathrm{red}}
\newcommand{\vir}[1]{[#1]^{\mathrm{vir}}}

\DeclareMathOperator{\sym}{Sym}

\DeclareMathOperator{\spec}{Spec}
\DeclareMathOperator{\id}{\mathrm{Id}}

\newcommand{\gm}{\mathbb{G}_\text{m}}

\newcommand{\f}{\mathrm{f}}
\newcommand{\m}{\mathrm{m}}

\newcommand{\ob}{\mathrm{ob}}

\newcommand{\perf}{\mathbf{Perf}}
\newcommand{\Pic}{\mathbf{Pic}}

\newcommand{\qcoh}{D_\mathrm{qcoh}}
\newcommand{\coh}{D_\mathrm{coh}}
\newcommand{\qc}{\mathrm{QCoh}}

\newcommand{\Hom}{\mathrm{Hom}}

\title{Virtual classes of $\gm$-gerbes}
\author{F. QU}
\email{fengquest@gmail.com}
\thanks{The author is supported by NSFC 
Young Scientists Fund 11801185}


\begin{document}
\maketitle
\begin{abstract}

We show that a perfect obstruction theory for a $\gm$-gerbe determines a semi-perfect obstruction theory for its base, which is perfect if the gerbe is quasi-compact and affine-pointed.  These results streamline the construction of a semi-perfect obstruction theory for the base, and allow us to relate virtual classes of the gerbe and its base.
\end{abstract}
\section{Introduction}

We discuss virtual classes of $\gm$-gerbes with perfect obstruction theories.
These gerbes appear naturally in the Donaldson-Thomas (DT) theory of 
smooth projective 3-folds. As moduli stacks of Bridgeland stable objects, $\gm$-gerbes are used to
 extract invariants when there are no strictly semi-stable objects.
 See e.g., \cite{BLMNPS} and references therein.
 In the presence of strictly semi-stable objects, the approach using virtual cycles to define
 generalized DT invariants is to associate
 some $\gm$-gerbe  intrinsically to
 a (derived) moduli stack of Bridgeland semi-stable objects. See e.g., \cite{KLS}. 

Recently virtual classes of Artin
 stacks are defined unconditionally (\cite{AP, Kh}) using higher categorical ingredients.
$\gm$-gerbes are probably the simplest Artin stacks, and their virtual classes can also be treated using the more classical approach of \cite{Po}, thus examining their virtual cycles is a natural step to take towards understanding examples.

Consider a $\gm$-gerbe $\cG$ with an absolute perfect obstruction theory over a DM stack $B$. The main observation is that
after truncating its perfect obstruction theory from $[-1,1]$ to $[-1,0]$, we can decompose 
the truncation into
moving and fixed parts. 
The moving part is given by a locally free sheaf of finite rank $H$ in degree $-1$, and the fixed part
determines a semi-perfect obstruction theory for $B$.
The virtual class of $\cG$ is obtained by pulling back the virtual class of $B$ then cap with the
Euler class of the vector bundle associated to $H$. When $B$ is quasi-compact and affine-pointed, the semi-perfect obstruction theory for $B$ is actually a perfect obstruction theory.

Obstruction theories for $B$ have been constructed in \cite{CL, HT}, and our results
come from efforts to formulate those constructions using the perfect obstruction theory 
for $\cG$. These results are not hard to prove, 
and details can be found in Section 3.
The two key ingredients recalled in Section 2.2 are decomposition
of quasi-coherent sheaves on $\cG$ into direct summands indexed by the characters of $\gm$ and the equivalence between the derived category of complexes of $\cO_\cG$-modules with quasi-coherent cohomology sheaves and 
the derived category of
quasi-coherent sheaves.

As an application, at the end of Section 3 we remark that two choices of
fixing determinant of perfect complexes produce the same semi-perfect obstruction theory.

\section{Preliminaries}
We work over the field of complex numbers $\bC$.
\subsection{Notation}
For an algebraic stack $X$, $\mathrm{QCoh}(X)$ denotes the abelian category of quasi-coherent sheaves on $X$, $D(\qc(X))$ its derived category,
$\qcoh(X)$ the derived category of $\cO_X$-modules with quasi-coherent cohomology sheaves, and 
$\coh(X)$ the derived category of $\cO_X$-modules with coherent cohomology sheaves. 
Here sheaves of $\cO_X$-modules are defined on the lisse-\'etale site of $X$.

For derived categories, superscripts are used to further specify
the range of cohomology sheaves.
The truncation functor for complexes is denoted by $\tau$, with a superscript to indicate the range of a truncation.
A complex is perfect in $[a, b]$ if it is locally (in the lisse-\'etale topology) quasi-isomorphic to a complex of locally free sheaves of finite rank in degrees $[a,b]$.
The derived pullback of derived category objects along a map $f$ is also denoted by $f^*$. The superscript $\vee$ denotes taking dual.

\subsection{Quasi-coherent sheaves on $\gm$-Gerbes and  their derived categories}

Let $B$ be a DM stack  locally of finite type over $\bC$, and $p\colon \cG \to B$ a $\gm$-gerbe over $B$.  Any quasi-coherent  sheaf $F$ on $\cG$ has  a decomposition 
$F=\oplus_{ i \in \bZ} F_i$ where $F_i$ has weight $i$.(See e.g., \cite[Proposition 2.2.1.6]{Li2}.)

\begin{rmk}
On the trivial gerbe $U \times B\gm$, 
a quasi-coherent sheaf with weight $i$ is of the form
$F\boxtimes \cO(i)$,where $F$ is a quasi-coherent sheaf on $U$
and  $\cO(i)$ the line bundle on $B\gm$ induced by
the character of $\gm$ with weight $i$.
\end{rmk}

Following \cite{GP}, we  call $F_0$ the fixed part of $F$ and  $\oplus_{i \ne 0} F_i$ the moving part, denoted by $F^\f$ and $F^\m$ respectively.

We have a  decomposition of abelian categories
\[
\qc(\cG) \simeq \qc(\cG)^\f \times  \qc(\cG)^\m,
\]
where 
$\qc(\cG)^\f$
(resp.
$\qc(\cG)^\m$) is the full subcategory of quasi-coherent sheaves with only fixed (resp. moving) part.
The pushforward $p_*$ induces an equivalence \[\mathrm{QCoh}(\cG)^\f \simeq\mathrm{QCoh}(B)\] with inverse $p^*$.

For an algebraic stack $X$, the inclusion map 
$\mathrm{QCoh}(X) \to \mathrm{Mod}(\cO_X)$ from the category of quasi-coherent sheaves on $X$ into the category of $\cO_X$-modules induces a map between derived categories.
\begin{prop}[{\cite[Theorem C.1]{HNR}}]
Let $X$ be an algebraic stack. If $X$ is either quasi-compact with affine diagonal or noetherian and affine-pointed, then the natural map between derived categories
$D^+(\mathrm{QCoh}(X)) \to \qcoh^+(X)$ is an equivalence. 
\end{prop}

Recall $X$ is affine pointed if for every morphisms $\spec k \to X$ from a field $k$ to $X$ is affine. For instance, if $X$ has affine diagonal, then it is affine-pointed.

\subsection{Perfect obstruction theory (POT)}

\begin{defn}[\cite{BF1, No,Po}]
Let $f \colon X \to  Y$ be a map  between algebraic stacks locally  of finite type over $\bC$,  an obstruction theory for $f$ is  a  map $\phi\colon E^\bullet  \to L_f$ in $\qcoh^{\le 1}(X)$ such that  $h^0(\phi), h^1(\phi)$ are isomorphisms and $h^{-1}(\phi)$ surjective in 
$\mathrm{QCoh}(X)$. Here $L_f=\tau^{\ge -1}\bL_f $ is the truncated cotangent complex, with $\bL_f \in D_{\mathrm{coh}}^{\le 1}(X)$ being the cotangent complex of $f$.

The obstruction theory $\phi$ is perfect if $E^\bullet$ is perfect in $[-1,1]$.

The obstruction sheaf $\ob_\phi$ is defined as $h^1({E^\bullet}^\vee)$.
\end{defn}

\begin{rmk}
 For a summary of $\bL_f$, see e.g., \cite[2.4]{AOV}. 
 POTs defined using $\bL_f$ in place of $L_f$ can be truncated to $[-1,1]$ and give rise to POTs defined above. 
 
We use $L_f$ in this paper for technical reasons. In Lemma \ref{van},  to ensure certain vanishings we need that $L_B$ belongs to $\qcoh^{\ge -1}(\cG)$.  And in the proof of 
Theorem \ref{pot}, the assumptions there allows us to replace
 $\qcoh^{\le 1}(X)$ by $D^{\le 1}(\qc(X))$.
\end{rmk}

When the map $f$ is DM, $L_f \in \coh^{[0,1]}(X)$, and a POT for $f$ induces a closed embedding
of the intrinsic normal sheaf $\cN_f=h^1/h^0(L_f^\vee)$ 
\footnote{Here $L_f$ also denote its extension to the big  fppf site, see \cite[Section 2]{BF1}.
} 
into the vector bundle stack 
$\cE_f=h^1/h^0({E^\bullet}^\vee)$
(\cite[Theorem 4.5]{BF1}
\cite[Theorem 2.3]{Po}).
The intrinsic normal cone $\cC_f$ is a closed substack of $\cN_f$, and 
the virtual class is then defined  by intersecting  $\cC_f$
viewed as a cycle of $\cE_f$ with its zero section.

\begin{rmk}
When $f$ is not DM, similar to \cite{BF1}, $L_f$ determines  a Picard 2-stack $\cN_f$, 
 and a POT $\phi$ induces a closed embedding of $\cN_f$ into the vector bundle 2-stack associated to $E^\bullet$ (\cite{AP, Kh}). To define a virtual class, 
we only need a closed embedding 
$
\pi_0(\cN_f) \to \pi_0(\cE_f)
$
between their coarse sheaves (See e.g., \cite[Section 2]{Be}), this observation goes back to 
\cite{LT} and is recasted into semi-perfect obstruction theory. In particular, only the truncation of $\phi$  to $[-1,0]$ should matter, and this is the earlier approach of \cite{No,Po}.
\end{rmk}

\begin{defn}[\cite{CL}]
 Let $X$ be a DM stack over a pure dimensional base $S$, a semi-perfect obstruction theory for $X$ over $S$ consists of a collection of \'etale locally defined POTs  $(U_i, \phi_i)$, where $\{U_i\}$ is an \'etale cover of $X$ and $\phi_i \colon E_i^\bullet \to L_{U_i/S}$ is a POT for $U_i$,  and they satisfy the following  conditions (1) and (2).
\begin{enumerate}
\item the local obstruction sheaves $\ob_i=h^1({E_i^\bullet}^\vee)$ are isomorphic over $U_{ij}$
and descend to an obstruction sheaf $\ob$ on $X$.
\item
the restrictions of $\phi_i$ and $\phi_j$ to $U_{ij}$ give the same obstruction assignment under the identification of local obstruction sheaves in $(1)$.
\end{enumerate}
\end{defn}

By the proof of \cite[Proposition 2.1]{CL}, 
an equivalent form of Condition (2) is the following condition (2').
For any closed point $x\colon \spec\bC \to X$, there exists a well-defined map
\[
 h^1((x^*L_{X/S})^\vee) \to x^*\ob
\] 
obtained using 
a factorization of $x$ as the composition of $y \colon \spec\bC \to U_i$ and  $U_i \to X$.
For any choice of $y$, the POT $\phi_i$ induces a map 
\[
h^1((y^*\phi_i)^\vee) \colon h^1((y^*L_{U_i/S})^\vee) \to h^1((y^*{E^\bullet_i})^\vee) \simeq y^*\ob_i,
\] which can be identified as a map
\[
 h^1((x^*L_{X/S})^\vee) \to x^*\ob.
\]

 \begin{rmk}
Conditions (1) and (2) are related to maps between coarse sheaves. The locally defined POT $\phi_i$ corresponds to a closed imbedding
$\cN_i \to \cE_i$, 
where $\cN_i$ denotes
the intrinsic normal sheaf of $U_i/S$, and
$\cE_i$ the vector bundle stack associated with $E_i^\bullet$.
We have induced maps between coarse sheaves
\[
\pi_0(\cN_i) \to \pi_0(\cE_i) \simeq \ob_i.
\footnote{
Here $\ob_i$ denote its extension to the big \'etale site of $U_i$ as in \cite{Be}.
}
\] 
Denote $\cN_i^\red$ the reduced stack associated to $\cN_i$, the composition of  the closed embedding $\cN^\red_i \to \cN_i$ with
the embedding  $\cN_i \to \cE_i$ induces 
\[
\pi_0(\cN_i^\red) \to \pi_0(\cE_i) \simeq \ob_i.
\]
It is clear from the proof of \cite[Proposition 2.1]{CL}
that 
if the maps $\pi_0(\cN_i) \to  \ob_i
$ descend to $\pi_0 (\cN_{X/S}) \to \ob
$, then conditions (1) and (2) are satisfied, and together (1) and (2) imply that the maps
$
\pi_0(\cN_i^\red) \to \ob_i
$
 descend to 
$
\pi_0 (\cN_{X/S}^\red) \to \ob
$. 

\begin{rmk}
The closed imbedding $\pi_0 (\cN_{X/S}^\red) \to \ob$ is used to construct the virtual cycle. It
determines a map between Chow groups
$A_*(\cN_{X/S}) \simeq A_*(\cN_{X/S}^\red) \to A_*(\ob)$.
The virtual class is obtained by pushing forward the cycle 
determined by the intrinsic normal cone $C_{X/S}$ inside $\cN_{X/S}$ to a cycle class of the coherent sheaf stack $\ob$, then Gysin pullback to $X$. See \cite{CL} for the original construction, and \cite{Lee} from the virtual pullback viewpoint.
\end{rmk}

\end{rmk}

\begin{defn}[\cite{KL1,Ji}]
Let $X$ be a DM stack over $\bC$, a semi-perfect obstruction theory for $X$ over $\bC$ is symmetric if its local POTs  $\phi_i$ are symmetric (\cite{Be,BF2}) and the induced isomorphisms $\ob_i \simeq \Omega_{U_i}$ descend to
an isomorphism $\ob \simeq \Omega_X$.

\end{defn}
\begin{rmk}
Symmetric semi-perfect obstruction theory is introduced first for analytic spaces in \cite{KL1}, its adaptation to DM stacks appeared in \cite{Ji}.
\end{rmk}

 For any \'etale map $U \to X$, the pullback of a POT $\phi$ for $X$ over $S$ to $U$ determines a POT for $U$ over  $S$. 
 In this way,  a POT $\phi$ for $X/S$ induces a semi-perfect obstruction theory.

\begin{rmk}
Determining whether a semi-perfect obstruction theory comes from a
POT requires additional information.
If a semi-perfect obstruction theory is induced from a POT on $X$ then the vector bundle stacks
$\cE_i=h^1/h^0({E_i^\bullet}^\vee)$
 descend to $X$.
 In general, the descent data for $\cE_i$ include an isomorphism
 $a_{ij}\colon {\cE_i}_{|_{U_{ij}}} \simeq {\cE_j}_{|_{U_{ij}}}$ on each $U_{ij}$, a two arrow 
 $b_{ijk}$ between ${a_{ik}}_{|_{U_{ijk}}}$ and ${a_{ij}}_{|_{U_{ijk}}} \circ {a_{jk}}_{|_{U_{ijk}}}$ on each $U_{ijk}$, and compatibilities between $\{b_{ijk}\}$ on each $U_{ijkl}$. In terms of complexes, $b_{ijk}$ correspond to chain homotopies, which are invisible in the derived category. 
\end{rmk}

\subsection{Comparison between Gysin pullbacks along zero sections}
We recall a standard result comparing Gysin pullbacks along zero sections
of vector bundle stacks in an exact sequence.

Let $Y$ be an algebraic stack of finite type over $\bC$, and $f\colon \cE \to \cE''$ a smooth surjective map between vector bundle stacks over $Y$. Let  $\cE'$ be the kernel of $f$, so that we
have a cartesian diagram
\begin{equation}\label{cart}
\xymatrix{
\cE' \ar[r]\ar[d]  & Y\ar[d]^{0_{\cE''}}\\
\cE \ar[r]^f        & \cE'',
}
\end{equation}
where $0_{\cE''}$ is the zero section of $\cE''$.
When $\cE''$ is a vector bundle, the map  $0_{\cE''}$ is  a closed embedding, then
$\cE' \to \cE$ is also a closed embedding.

\begin{lemma} \label{top}
Assume $\cE''$ is an vector bundle
and $Y$ is stratified by global quotients.
Denote $i$ the closed embedding of $\cE'$ into $\cE$. 

As  maps between Chow groups,
\[
0_{\cE}^!\circ {i_*} =e(\cE'') \circ 0_{\cE'}^! .
\] Here  $i_*$ denotes the pushforward  map,  $0^!_\cE$ and $0^!_{\cE'}$ Gysin pullbacks along zero  sections, and $e(\cE'')$ the Euler class.

\end{lemma}

\begin{proof}
We include a proof for convenience of the reader.

Let $\pi_\cE, \pi_{\cE'}$ denote the projections $\cE \to Y$ and $\cE' \to Y$ respectively.
As $Y$ is of finite  type and stratified by global quotients, the flat pullback  $\pi_{\cE'}^!$
is an  isomorphism with inverse $0_{\cE'}^!$.

Consider  $0_{\cE}^!\circ {i_*}\circ \pi_{\cE'}^!$.   As diagram \eqref{cart} is cartesian,
we have 
\[
{i_*}\circ \pi_{\cE'}^! =f^! \circ {0_{\cE''}}_*.
\]
Therefore
\[
0_{\cE}^!\circ {i_*}\circ {\pi_\cE}^! = 0_{\cE}^! \circ f^! \circ {0_{\cE''}}_*
= (f\circ 0_\cE)^!\circ  {0_{\cE''}}_* = 0_{\cE''}^!\circ {0_{\cE''}}_* =e(\cE''),
\] and this is equivalent  to  $0_{\cE}^!\circ {i_*} =e(\cE'') \circ0_{\cE'}^! $.

\end{proof}

\section{ virtual class of $\gm$ gerbes}
In this section, we prove the results sketched in the introduction.

Let $p\colon \cG \to B$ be a $\gm$-gerbe over  a DM stack $B$ of finite type over $\bC$, and $\phi\colon E^\bullet  \to L_{\cG}$ a POT for $\cG$.
\subsection{Virtual class of $\cG$}
For the distinguished triangle between cotangent complexes
$
p^*\bL_B  \to \bL_\cG \to \bL_p, 
$ we have
$p^*\bL_B  \simeq \tau^{\le 0}(\bL_\cG)$
 and $h^1(\bL_\cG) =h^1(\bL_p)$, since
 $\bL_B \in \qcoh^{\le 0}(B)$ and 
$\bL_p[1] \simeq h^1(\bL_p)$ is locally free.
Therefore we have a distinguished triangle between truncated cotangent complexes
\[
p^*L_B  \to    L_\cG            \to \bL_p.
\]

As in \cite{Po}, we truncate $E^\bullet$ to $[-1,0]$. 
There exists a map between distinguished triangles
\[
\xymatrix{
F^\bullet \ar[r]\ar[d] & E^\bullet \ar[r]\ar[d]^\phi & \bL_p\ar[d]^{\id} \\
p^*L_B  \ar[r]        &  L_\cG \ar[r]           &\bL_p,
}
\] where $E^\bullet \to \bL_p$ is the composition of $\phi$ and  $L_\cG \to \bL_p$.
Denote $\psi$ the first vertical map $ F^\bullet  \to p^*L_B$, then it can be identified with $\tau^{\le 0}(\phi)$. Note that $h^0(\psi)$ is an isomorphism, $h^{-1}(\psi)$ is surjective, and $F^\bullet$  is perfect  in $[-1,0]$.

Then we can remove the moving part of $h^{-1}(F^\bullet)$ from $\psi$ as follows.
The  map 
\[
h^{-1}(F^\bullet)[1]\simeq \tau^{\le -1}F^\bullet \to F^\bullet
\] and the inclusion map $h^{-1}(F^\bullet)^\m \to h^{-1}(F^\bullet) $
induces a map
\[
 h^{-1}(F^\bullet)^\m[1]\to h^{-1}(F^\bullet)[1]\to F^\bullet.
\] Denote ${F^\bullet}^\f$ the cone of the map
$h^{-1}(F^\bullet)^\m[1]\to F^\bullet$. By the lemma below we see that $\psi$ uniquely induces a map
\[
\psi^\f \colon {F^\bullet}^\f\to p^*L_B.
\] Note that  
$h^0(\psi^\f)$ is an isomorphism, and $h^{-1}(\psi^f)$ is surjective.

\begin{lemma}
There are no nonzero maps from $h^{-1}(F^\bullet)^\m[1]$ or $h^{-1}(F^\bullet)^\m[2]$
to $ p^*L_B$ in $\qcoh(\cG)$.
\end{lemma}

\begin{proof} \label{van}
As $L_B \in  \qcoh^{[-1,0]}(B)$ and $p$ is flat,  $p^*L_B\in \qcoh^{[-1,0]}(\cG)$.  Since $h^{-1}(F^\bullet)^\m[2] \in \qcoh^{\le -2}(\cG)$ and $p^*L_B \in 
\qcoh^{\ge -1}(\cG)$,
$
\Hom_{\qcoh(\cG)}(h^{-1}(F^\bullet)^\m[2], p^*L_B)=0.
$

Similarly, using the exact triangle $h^{-1}(p^*L_B)[1] \to p^*L_B \to h^0(p^*L_B)$, we see that 
\[
\Hom_{\qcoh(\cG)}(h^{-1}(F^\bullet)^\m[1], p^*L_B) \simeq
\Hom_{\qcoh(\cG)}(h^{-1}(F^\bullet)^\m[1], h^{-1}(p^*L_B)[1]),
\]
which can be identified with
$\Hom_{\textrm{QCoh}(\cG)}(h^{-1}(F^\bullet)^\m, h^{-1}(p^*L_B))
$.
As $h^{-1}(p^*L_B) \simeq p^*h^{-1}(L_B)$ has no moving part, there is only the zero map 
between $h^{-1}(F^\bullet)^\m $ and $h^{-1}(p^*L_B)$.
\end{proof}

\begin{lemma}
The sheaf $h^{-1}(F^\bullet)^\m$ is locally free of finite rank, and 
the complex ${F^\bullet}^\f$ is perfect in $[-1,0]$.
\end{lemma}

\begin{proof}

Locally, we can represent $F^\bullet$ by a two term complex $\widetilde{F}^\bullet$ of locally free sheaves of  finite rank in $[-1,0]$. The complex $\widetilde{F}^\bullet$ is the direct sum of its moving and fixed parts, which are also complexes of locally free sheaves in $[-1,0]$.  

It is clear 
$h^{-1}(F^\bullet)^\m$ is given by $h^{-1}$ of the moving part of $\widetilde{F}^\bullet$.
As $h^0(F^\bullet) \simeq h^0(p^*L_B)$ has no moving part, the moving part 
of
$\widetilde{F}^\bullet$   has vanishing $h^0$. Therefore $ h^{-1}(F^\bullet)^\m$ is quasi isomorphic to the moving part of $\widetilde{F}^\bullet$, and as  
 the kernel of a surjective map between locally frees of finite rank, it is locally free of finite
 rank.

Now by the construction of ${F^\bullet}^\f$ as a cone, we conclude ${F^\bullet}^\f$ is locally represented by 
the fixed part of $\widetilde{F^\bullet}$, and it is perfect in $[-1,0]$.
\end{proof}

The map $\psi^\f$ induces a closed imbedding  
\[
\iota \colon p^*\cN_B\to \cF^\f
\] 
where $\cN_B=h^1/h^0(L_B^\vee)$
is  the intrinsic normal  sheaf  of $B$  and 
$\cF^\f=h^1/h^0({{F^\bullet}^\f}^\vee)$ the vector bundle stack associated with  ${F^\bullet}^\f$. Let $\cC_B \subset \cN_B$ be the intrinsic normal cone of $B$, then we can view $p^*\cC_B$
as a closed substack of $\cF^\f$ via $\iota$.

\begin{prop}\label{gvir}  Assume $B$ is of finite type over  $\bC$ and  stratified by global quotients.
The virtual class $\vir{\cG}$ determined by $\phi$ 
is given by 
\[
e(\bV(h^{-1}(F^\bullet)^\m)) \cap 0^![p^*\cC_B]
\] 
here $0^!\colon  A_*(\cF^\f) \to A_*(\cG)$ denotes Gysin pullback along
the zero section of $\cF^\f$, 
$\bV(h^{-1}(F^\bullet)^\m)$ the vector bundle $\spec \sym {h^{-1}(F^\bullet)^\m}$, and $e$ the Euler class.
\end{prop}
\begin{proof}
This follows from the definition of virtual classes in \cite{Po} and the construction of $\psi^\f$.

The map $\psi$ induces a closed embedding $p^*\cC_B \to \cF$, where
$\cF$ denotes  $h^1/h^0({F^\bullet}^\vee)$. The virtual class 
of $\cG$ is obtained by pulling back $[p^*\cC_B]$ along the zero section
of $\cF$. 

The exact triangle 
 \[
  h^{-1}(F^\bullet)^\m[1]\to F^\bullet \to {F^\bullet}^\f 
 \]
 induces 
an exact sequence of vector bundle stacks
\[
 \cF^\f  \to \cF \to \bV(h^{-1}(F^\bullet)^\m).
\]  In particular, we have an cartesian square
 \[
 \xymatrix{
 \cF^\f \ar[r] \ar[d] & \cG \ar[d]\\
 \cF  \ar[r]  & \bV(h^{-1}(F^\bullet)^\m)
}
\]  for which the assumptions in Lemma \ref{top} are satisfied.
Here   the right vertical map from $\cG$ is the zero section of  $\bV(h^{-1}(F^\bullet)^\m)$.
By the construction of $\psi^\f$, we see that the closed embedding $p^*\cC_B \to \cF$ factor through $\cF^\f$ as it is the case locally, then we conclude the proof by Lemma \ref{top}.
\end{proof}

\begin{rmk} 
As $\cG$ has affine stabilizers, it is stratified by global quotients by \cite[
Proposition 3.5.9]{Kr} and $0^!\colon  A_*(\cF^\f) \to A_*(\cG)$ is defined.
\end{rmk}

\subsection{Semi-perfect obstruction theory for $B$}

For any \'etale map $U \to B$ with section $s\colon U \to \cG$, 
it is straightforward to see that $s$ is smooth, $s^*p^*L_B\simeq L_U$, and $s^*\psi^\f$ is a POT for $U$, which we shall denote by  $\psi_U$.

\begin{prop}
$\{\psi_U\}$ determines 
a semi-perfect obstruction theory $\psi_B$ for $B$.
\end{prop}

\begin{proof}
We verify conditions (1) and (2') for semi-perfect obstruction theories.

The local obstruction sheaf on $U$ is obtained by 
the pullback along  $s$ of $h^1({{F^\bullet}^\f}^\vee)$.
As locally ${F^\bullet}^\f$ is given by a complex of locally free sheaves, and its cohomology sheaves have no moving parts, 
we see that $h^1({{F^\bullet}^\f}^\vee)$ has no moving part. Then (1)
is satisfied with $\ob =p_*h^1({{F^\bullet}^\f}^\vee)$.

Condition (2') is satisfied because there is a unique section over any closed point of $B$.
\end{proof}

\begin{rmk}
If there exists a  nondegenerate symmetric bilinear pairing $F^\bullet \simeq {F^\bullet}^\vee[1]$, then  $\psi^\f=\psi$ and $\psi_B$ is symmetric.
\end{rmk}

\begin{thm}Assume $B$ is proper over $\bC$ and stratified by global quotients. Denote $\vir{B}$ the virtual class determined by $\psi_B$. 
We have
\[
\vir{\cG}= e(\bV(h^{-1}(F^\bullet)^\m))
\cap p^*\vir{B}.
\]
\end{thm}

\begin{proof}
We show that $
p^*\vir{B}=0^![p^*\cC_B],
$ then the theorem follows from Proposition \ref{gvir}.

Recall $\psi$ induces a closed imbedding 
$
\iota\colon p^*\cN_B \to \cF^\f.
$
And for each $U$ \'etale over $B$ with section $s\colon U \to \cG$, the POT $\psi_U$
is determined by $s^*\iota$.  
The semi-perfect obstruction theory $\psi_B$
determines a map 
$
\pi_0(\cN_B^\red) \to \ob,
$
and its pullback along $p$ can be identified with the map 
\[
\pi_0(p^*\cN_B^\red)\to \pi_0(\cF^\f)\simeq p^*\ob
\] induced by $\iota$. This is true because for any $U$ \'etale over $B$ with a section $s \colon U  \to \cG$,  the two maps match when pulled back along $s$ to $U$, 
and all possible $U$ form a smooth cover of $\cG$. Then it is easy to see 
$
p^*\vir{B}=0^![p^*\cC_B]
$ from the construction of $\vir{B}$ in \cite{CL}, since the flat pullback $p^*$ commutes with
the operations used in defining $\vir{B}$, this also follows from bivariance of 
virtual pullbacks established in \cite{Lee}.
\end{proof}

\begin{rmk}

Since $\cO_\cG$ has no moving parts, 
a cosection $\ob_\phi \to \cO_\cG$ is 
equivalent to a map $p^*\ob \to \cO_\cG$,  which
by  adjunction is the same as a 
cosection $\ob \to \cO_B$.
The theorem also holds for localized virtual cycles (\cite{KL, Ki}) by the same argument.
\end{rmk}

\begin{thm} When $B$ is quasi-compact and affine-pointed,   \label{pot}
the semi-perfect obstruction theory $\psi_B$ is a POT.
\end{thm}

\begin{proof}
It is easy to check the condition of being quasi-compact and affine pointed
holds for $B$ if and only if it holds for $\cG$.

As now 
$
D^+(\mathrm{QCoh}(\cG)) \simeq \qcoh^+(\cG),
$ we view $\psi$ as a map in $D^+(\mathrm{QCoh}(\cG))$. Under the equivalence
$
D^+(\mathrm{QCoh}(\cG)) \simeq D^+(\mathrm{QCoh}(\cG)^\f) \times 
D^+(\mathrm{QCoh}(\cG)^\m),
$
the component in $D^+(\mathrm{QCoh}(\cG)^\f)$ of $\psi$ is $\psi^\f$, and $\psi_B$ is obtained from $\psi^\f$
under the equivalences $D^+(\mathrm{QCoh}(\cG)^\f) \simeq D^+(\mathrm{QCoh(B)}) \simeq \qcoh^+(B)$.
\end{proof}

\begin{rmk}
For gerbes banded by the cyclic group $\mu_r$ with POTs, their POTs have no $h^1$ and we can also decompose them into moving and fixed parts 
with weights in $\bZ/r\bZ$, then all results above hold.
\end{rmk}

\subsection{}

Let $X$ be a smooth projective 3-fold over $\bC$, the moduli stack $\cG$ of perfect complexes which are simple and 
without higher automorphisms 
\footnote{
Negative self $\mathrm{Ext}$ groups are zero.
}
is a
$\gm$-gerbe locally of finite type (\cite[Corollary 4.3.3]{Li1}),
and has a derived enhancement (\cite[Definition 5.1]{STV}) , which has a (-1)-symplectic  structure (\cite{PTVV}) when $X$ is Calabi-Yau. These derived enhancements induce obstruction theories (\cite[Proposition 1.2]{STV}).
Alternatively, obstruction theories can be constructed as in \cite{CL} using the deformation-obstruction results in \cite{HT}.  If 
truncated obstruction theories are perfect, results in this section apply, and hopefully complement the perspectives in existing literature, e.g., \cite{CL, PR, Th}.

To work with open substacks of $\cG$ which are of finite type, we need to fix some numerical invariants and a stability condition. (See e.g.,\cite{BLMNPS}.) To define DT type invariants, we need to remove $\gm$ from the automorphism group of stable objects, which can be achieved by fixing the determinant of complexes. There are two choices to fix the determinant as  discussed below. Using the results we obtained, it is easy to see both choices produce the same semi-perfect obstruction theory.

Denote $\perf(X)$ the stack of perfect complexes on $X$ (\cite{STV, TV}),
$\Pic(X)$ the Picard stack of $X$, $[\Pic(X)/B\gm]$
\footnote{
See \cite[C.3]{AGV} for viewing rigidification as taking quotient.
}
 the Picard scheme of $X$, and let $L$ be a line bundle on $X$.
 Note that $\cG$ is open in $\perf(X)$.
Consider the cartesian diagrams
\[
\xymatrix{
\cG_L        \ar[r]\ar[d]        &  \spec \bC             \ar[d]    &\\
 \cG_L'          \ar[r]  \ar[d]  &B\gm    \ar[r] \ar[d]     & \spec \bC\ar[d]\\
\cG \ar[r] & \Pic(X)  \ar[r]  &  [\Pic(X)/B\gm],
} 
\] where $\cG_L$ 
and $\cG_L'$ denote fiber products, the vertical arrows from $\spec \bC$ are determined by $L$, and $\cG \to \Pic(X)$ induced by the perfect determinant morphism (\cite[Definition 3.1]{STV}).  Note that $\cG_L$ is a gerbe banded by cyclic groups and 
 $\cG_L'$ a $\gm$-gerbe, and
 the map $\cG_L \to \cG_L'$ induces an isomorphism after rigidification.
  
  The obstruction theory for $\cG \to \Pic(X)$ comes from the tangent complex of
$\bR\perf(X) \to \bR\Pic(X)$(\cite[Proposition 3.2]{STV}). By base change, we obtain obstruction theories for $\cG_L \to \spec \bC$
and $\cG_L' \to B\gm$. From the obstruction theory of $\cG_L' \to B\gm$, we obtain an obstruction theory of $\cG_L'$(\cite[Appendix B]{GP}). As the tangent complex of $B\gm$ is perfect in degree $-1$, the obstruction theory for
$\cG_L$ is perfect if and only if the obstruction theory for $\cG_L'$ is, and in that case, the induced semi-perfect obstruction theories on
their rigidifications are identical. In fact, the truncation of $\phi$ to $\psi$ in the beginning of the section reverses the process of obtaining the POT for $\cG_L'$ from $\cG_L' \to B\gm$.

\end{document}